\documentclass[11pt,a4paper]{article}
\usepackage[T1]{fontenc}
\usepackage{lmodern}
\usepackage[margin=28mm]{geometry}
\usepackage{amsmath,amssymb,amsthm}
\usepackage{microtype}
\usepackage{xurl}
\usepackage[colorlinks=true,linkcolor=blue,citecolor=blue,urlcolor=blue]{hyperref}
\newtheorem{theorem}{Theorem}[section]
\newtheorem{proposition}[theorem]{Proposition}
\newtheorem{lemma}[theorem]{Lemma}
\theoremstyle{remark}
\newtheorem{remark}[theorem]{Remark}
\numberwithin{equation}{section}
\allowdisplaybreaks[1]
\hypersetup{pdftitle={An explicit elementary proof of no Landau--Siegel zeros},pdfauthor={Victor Zhenyu Guo},pdfsubject={An explicit elementary proof using integer counting, paired character tails, and weighted positivity}}

\title{An explicit elementary proof of no Landau--Siegel zeros}
\author{Victor Zhenyu Guo\\[0.5em]
{\small School of Mathematics and Statistics, Xi'an Jiaotong University}\\
{\small Xi'an, Shaanxi, China.}\\
{\small \href{mailto:vzguo@foxmail.com}{\texttt{vzguo@foxmail.com}};
\href{mailto:guozyv@xjtu.edu.cn}{\texttt{guozyv@xjtu.edu.cn}}}}
\date{}

\begin{document}
\maketitle

\begin{abstract}
We give an elementary proof of the absence of Landau--Siegel zeros
with the explicit constant $c=1/100$. For every primitive quadratic
Dirichlet character $\chi$ of conductor $q\ge3$, each real zero
$\beta\in(0,1)$ of $L(s,\chi)$ satisfies $(1-\beta)\log q>c$.
The proof combines integer coefficient counting, polynomial
congruences, a paired-tail hyperbola estimate, and a logarithmically
weighted comparison of positive sums. All required constants are
explicit, and one choice of parameters covers every conductor.
The argument uses elementary arithmetic and real summation and
integration, without separate zero computations for small conductors.
\end{abstract}



\section{Introduction}
\label{sec:introduction}

Our aim is to prove a Landau--Siegel zero gap by elementary methods
while retaining a fully explicit numerical constant. We obtain
$(1-\beta)\log q>c$ with $c=1/100$ in Theorem~\ref{thm:main},
using a single choice of parameters for every conductor $q\ge3$.
All estimates needed for this value of $c$ have specified constants
and remainders, including the estimates for small conductors.
The proof uses integer and polynomial arithmetic, partial summation,
and real integration.

The argument follows the Frobenius method of \cite{OpenAI2026}
and its subsequent variants
\cite{Banks2026,Kerr2026,Green2026,vanDijk2026,Stepanov2026}.
For an exposition of the polynomial-method framework, see also
Tao~\cite{Tao2026}.
The arithmetic part combines polynomial congruences with integer
coefficient counting; injectivity follows by counting relatively
prime quadratic factors modulo a prime. The real-variable part
pairs character tails before estimating them and then compares
logarithmically weighted positive sums.

Throughout, $\chi$ denotes a primitive quadratic Dirichlet character
of conductor $q\ge3$, unless otherwise stated; see
\cite[Sections~6.5, 6.8, and~8.7--8.9]{Apostol}.

The letter $p$ denotes a prime, logarithms are natural, and
\[
 S_-(x,\chi)=\sum_{\substack{p\le x\\\chi(p)=-1}}\frac{\log p}{p}.
\]
For $s>0$, we use
\[
 L(s,\chi)=\lim_{x\to\infty}\sum_{n\le x}\frac{\chi(n)}{n^s}
 \qquad(s>0).
\]
The series converges by bounded character sums and partial
summation \cite[Theorem~6.17]{Apostol}. The condition
$L(\beta,\chi)=0$ below refers to this series.

\begin{theorem}[An explicit real zero gap]
\label{thm:main}
Let $\chi$ be a primitive quadratic Dirichlet character of
conductor $q\ge3$. Every real zero $\beta\in(0,1)$ of $L(s,\chi)$
satisfies
\begin{equation}\label{eq:main-gap}
 (1-\beta)\log q>\frac1{100}.
\end{equation}
\end{theorem}

\begin{remark}[Choice of the constant]
\label{rem:constant}
The same elementary argument also yields $c=1/64$; a verification is
given in Appendix~\ref{app:constant}.
For $q\ge7$, take $x=q^{46}$ and $y=q^4$ in
Proposition~\ref{arith:main} and Lemma~\ref{real:large-prime}.
For $q=3,4,5$, grouping the character series into positive blocks
gives $L(s,\chi)>0$ for $s>0$, and no primitive character has
conductor $6$.
We retain $c=1/100$ so that a single choice of parameters covers
all conductors and keeps the numerical verification short.
\end{remark}

\subsection{Earlier results and the scope of the theorem}

Banks \cite[Theorem~1.1]{Banks2026} gives the constant $0.003$ for
$q\ge10^{10}$, and for all $q\ge3$ when combined with the
small-conductor zero computations cited there. Kimura
\cite[Theorem~1 and its proof]{Kimura2026} gives $0.0005811$ for all
$q\ge3$, $0.005965$ for $\log q\ge35.0131$, and $0.01015$ for
$\log q\ge10000$. Thus $1/100$ is larger than their stated
all-conductor constants, while Kimura's last bound is larger in its
range. Our proof uses a single choice of parameters for all $q\ge3$
and includes its numerical comparisons as rational inequalities.

The quasi-Riemann result of
\cite[Theorem~1.1]{OpenAIQRH2026} states that every Dirichlet
$L$-function is zero-free in $\operatorname{Re}s>7/8$, with the
principal-character pole at $s=1$ excepted. That result already implies
the stronger inequality
\[
 (1-\beta)\log q\ge\frac18\log q
 \ge\frac18\log3>\frac1{100}.
\]
Theorem~\ref{thm:main} quantifies the elementary Frobenius argument;
the stronger half-plane theorem is not an input.

The arithmetic coefficients obtained here are not optimal.
Green \cite[Theorem~1.1]{Green2026} obtains a leading coefficient
$1/2+O(1/\alpha)$ at $x=q^\alpha$ for sufficiently large even integers
$\alpha$, when $q$ is sufficiently large in terms of $\alpha$.
The stronger estimate attributed to Barreto in
\cite[Theorem~1]{Stepanov2026}, and also proved in
\cite[Theorem~1]{vanDijk2026}, is
\[
 S_-(x,\chi)\le
 \frac12\log x+\frac14\log q+\log\log(3x)+O(1),
\]
with an effective absolute implied constant. Our estimate below
instead specifies every coefficient and the absolute remainder,
and includes all small-parameter cases.

\subsection{Proof outline}

The arithmetic input is Proposition~\ref{arith:main}, which proves
\begin{equation}\label{eq:intro-arithmetic}
 S_-(x,\chi)\le
 \frac{11}{20}\log x+\frac{29}{60}\log q+6
 \qquad(x\ge2).
\end{equation}
For the parameters in \eqref{arith:parameters} and an auxiliary prime
${p_0} $ with $\chi({p_0} )=-1$, Lemma~\ref{arith:injectivity} proves that the
evaluation map \eqref{arith:evaluation} is injective. If $0\ne f$ were
in its kernel, divide its coefficients by their largest common power
of ${p_0} $. The resulting reduction
$g(t)=f(t,t^{p_0})\pmod{p_0}$ is nonzero. The pairwise relatively prime
quadratic factors then give the contradiction
\[
 \begin{gathered}
 \prod_{\substack{0\le u<m\\1\le v<m}}
       \bigl((t-u)^2-dv^2\bigr)^{s_0}
       \mid g(t)\pmod{p_0},\\
 2s_0m(m-1)\le\deg g<{p_0}s_0<2s_0m(m-1).
 \end{gathered}
\]
The adapted integer bases
\eqref{arith:adapted-basis} supply the prime-power divisibility
\eqref{arith:divisibility}. Lemma~\ref{arith:counting} then gives the
upper bound \eqref{arith:H-upper}; comparison with
\eqref{arith:H-lower}, followed by the remainder estimate
\eqref{arith:remainder-six} and the small cases in the proof of
Proposition~\ref{arith:main}, yields \eqref{eq:intro-arithmetic}.

For the real-variable part, put
\[
 a(n)=(1*\chi)(n),\qquad
 H_\beta(t)=\sum_{n\le t}\frac{a(n)}{n^\beta}.
\]
At a real zero $1/2\le\beta<1$, Proposition~\ref{real:hyperbola} gives
\begin{equation}\label{eq:intro-hyperbola}
 H_\beta(t)=\frac{L(1,\chi)}{1-\beta}t^{1-\beta}+E(t),
 \qquad
 |E(t)|\le\left(q+1+\frac12\log t\right)t^{1/2-\beta}.
\end{equation}
Estimating the two tails in \eqref{real:paired-tails} jointly gives
\eqref{real:paired-error}, avoiding the factor $(1-\beta)^{-1}$ that
would arise from estimating them separately.

Finally, Lemma~\ref{real:large-prime} converts this estimate into a
prime-sum comparison.
Writing $H_1(y)=\sum_{n\le y}a(n)/n$ and
$N(t)=\sum_{n\le t}a(n)\log n/n$, the injection
$(n,p)\mapsto np$ for $n\le y<p\le x$ gives
\eqref{real:weighted-injection}, namely
\[
 H_1(y)\sum_{y<p\le x}\frac{a(p)\log p}{p}
 \le N(xy)-N(y).
\]
The partial-summation identity \eqref{real:weighted-abel} and
positivity eliminate $L(1,\chi)/(1-\beta)$ from the comparison,
giving \eqref{real:prime-mass}--\eqref{real:comparison-error}.
With the fixed choices $y=q^6$ and $x=q^{70}$ in
\eqref{proof:parameters}, the assumption
$(1-\beta)\log q\le1/100$ yields the incompatible bounds
\eqref{proof:lower-bound} and \eqref{proof:upper-bound}.
The final numerical contradiction is \eqref{proof:final-certificate}.

\subsection{Sources, refinements, and elementary scope}

Coefficient counting follows Kerr
\cite[Lemmas~2--5]{Kerr2026}; the triangular basis and its derivative
divisibility follow Green \cite[Lemma~5.1]{Green2026}.
The auxiliary-prime nonvanishing argument is the one in
\cite[proof of Theorem~1, Nonvanishing]{vanDijk2026}; see also
the Stepanov-type formulation in \cite[Lemma~4]{Stepanov2026}.
We express the latter argument through relatively prime quadratic
factors and polynomial division modulo a prime. Thus the entire
arithmetic proof can be carried out with integer coefficients.

The positive-convolution method is classical; see
\cite[Exercise~58]{Tao2014} and \cite[Section~1, Lemma~5]{Tao2015}.
The harmonic normalization $H_1$ and the single-large-prime injection
already appear in \cite[proof of Proposition~3.5]{TaoTeravainen2022}.
Bounds linear in $1-\beta$ are also standard; for example,
\cite[Lemma~1]{Kerr2026} bounds the split-prime sum by a constant
multiple of $\log q+(1-\beta)(\log x)^2$.
Our quantitative refinements are
\eqref{eq:intro-arithmetic}, the uniform remainder in
Proposition~\ref{real:hyperbola}, and the complete error term in
Lemma~\ref{real:large-prime}. These give the stated constant without
a preliminary lower bound for $L(1,\chi)$ or an estimate for
$L'/L(1,\chi)$. The underlying counting, Frobenius, and
positive-convolution mechanisms are established methods.

The prerequisites are the Chinese remainder theorem, quadratic
reciprocity, polynomial arithmetic modulo a prime, and real
summation and integration. Standard facts are cited in the relevant
places. The proof uses no complex zero formula, analytic continuation,
or functional equation. Proofs are retained for the specific estimates
and constants used in \eqref{eq:main-gap}.

\section{Standard facts and explicit prime estimates}
\label{sec:preliminaries}

We use Euler summation, partial summation, and the Dirichlet hyperbola
identity in their standard forms
\cite[Theorems~3.1, 4.2, and~3.17]{Apostol}.
For the Chinese remainder theorem, polynomial congruences, and
quadratic reciprocity, see \cite[Sections~5.4--5.7 and Chapter~9]{Apostol}.

\paragraph{Quadratic characters.}
The standard classification of primitive quadratic characters
\cite[Section~1.4, Theorem~1.1]{Goldmakher} gives a fundamental
discriminant ${d_\chi}$ with $|{d_\chi}|=q$ and $\chi=({d_\chi}/\cdot)$.
In arithmetic terms, there is a squarefree integer $d\ne1$ such that
\[
 {d_\chi}=
 \begin{cases}
 d,&d\equiv1\pmod4,\\
 4d,&d\equiv2,3\pmod4.
 \end{cases}
\]
Hence $|d|\le q$ and
\begin{equation}\label{prelim:quadratic-reduction}
 \chi(p)=\left(\frac dp\right)
 \qquad\text{for every odd prime }p.
\end{equation}
Only this elementary consequence of the classification is needed.

\paragraph{Numerical bounds.}
We will use the elementary inequalities
\begin{equation}\label{prelim:numerical}
 \begin{gathered}
 \frac{13}{12}<\log3<\frac{11}{10},\qquad
 \log2<\frac7{10},\qquad \log5<\frac{13}{8},\\
 \frac{35}{18}<\log7<\frac{39}{20},\qquad \log100<5.
 \end{gathered}
\end{equation}
For example, these bounds follow from the real power series
\[
 \log x=2\sum_{j=0}^{m-1}\frac{z^{2j+1}}{2j+1}+R_m,\qquad
 0<R_m<\frac{2z^{2m+1}}{(2m+1)(1-z^2)},\qquad
 z=\frac{x-1}{x+1}\quad(x>1).
\]
The choices $(x,m)=(2,1),(3,2),(5,3),(7,9)$ suffice;
$\log100=2\log2+2\log5<5$.
Let
\[
 \vartheta(x)=\sum_{p\le x}\log p,\qquad
 \psi(x)=\sum_{n\le x}\Lambda(n),\qquad
 M(x)=\sum_{p\le x}\frac{\log p}{p},
\]
where $\Lambda$ is the von Mangoldt function. The standard
Chebyshev and Mertens estimates are recorded with the constants
needed below; compare \cite[Section~4.5]{Apostol} and
\cite[Lemma~5(a),(b)]{Kimura2026}. We include only the calculations
that fix the constants in the forms used here.

\begin{lemma}\label{prelim:prime-estimates}
For $x\ge2$,
\begin{align}
 \vartheta(x)\le\psi(x)&\le(\log4)x,
 \label{prelim:chebyshev-bound}\\
 M(x)&\le\log x+\log4.
 \label{prelim:prime-upper-bound}
\end{align}
Moreover,
\begin{equation}\label{prelim:odd-square-bound}
 \sum_{\substack{p\ {\rm prime}\\p\ {\rm odd}}}
       \frac{\log p}{p^2}<\frac7{20},
\end{equation}
and, for $x\ge y\ge100$,
\begin{equation}\label{prelim:mertens-interval-bound}
 M(x)-M(y)>\log(x/y)-2.
\end{equation}
\end{lemma}

\begin{proof}
The binomial-coefficient argument for Chebyshev's estimate
(cf.\ \cite[Section~4.5]{Apostol}) gives
\[
 \begin{aligned}
 \psi(2m)-\psi(m)
 &\le\log\binom{2m}{m}\le2m\log2,\\
 \psi(2m+1)-\psi(m+1)
 &\le\log\binom{2m+1}{m}\le2m\log2.
 \end{aligned}
\]
Induction yields \eqref{prelim:chebyshev-bound}.
For integers $n\ge2$, factorial prime factorization
\cite[Theorem~3.14]{Apostol} then gives
\[
 nM(n)\le\log(n!)+\vartheta(n)\le n\log n+n\log4.
\]
Taking $n=\lfloor x\rfloor$ proves
\eqref{prelim:prime-upper-bound}.

For the square-denominator sum, $f(t)=(\log t)/t^2$ is convex for
$t\ge3$, since $f''(t)=(6\log t-5)/t^4>0$. Every prime exceeding $7$
belongs to one of the progressions $11+6j$ and $13+6j$, $j\ge0$.
Convexity bounds each term by the average over its centered interval
of length $6$. Summing these intervals gives
\[
 \begin{aligned}
 \sum_{\substack{p\ {\rm prime}\\p\ {\rm odd}}}\frac{\log p}{p^2}
 &\le\frac{\log3}{9}+\frac{\log5}{25}+\frac{\log7}{49}
   +\frac16\int_8^\infty\frac{\log t}{t^2}\,dt
   +\frac16\int_{10}^\infty\frac{\log t}{t^2}\,dt\\
 &=\frac{\log3}{9}+\frac{\log5}{25}+\frac{\log7}{49}
   +\frac16\left(\frac{\log8+1}{8}+\frac{\log10+1}{10}\right)\\
 &<\frac{30607}{88200}<\frac7{20},
 \end{aligned}
\]
where the last line uses \eqref{prelim:numerical}.

For the interval estimate put $A(t)=\sum_{r\le t}\Lambda(r)/r$.
The factorial identity \cite[Theorem~3.12, equation~(17)]{Apostol}
$\log(n!)=\sum_{r\le n}\Lambda(r)\lfloor n/r\rfloor$,
\eqref{prelim:chebyshev-bound}, and integral comparison for $\log(n!)$
give
\begin{equation}\label{prelim:von-mangoldt-bounds}
 \log n-1+\frac1n
 \le {A}(n)\le
 \log n-1+\log4+\frac{1+\log n}{n}.
\end{equation}
The higher prime powers satisfy
\[
 \begin{aligned}
 V=\sum_p\frac{\log p}{p(p-1)}
 &\le\frac{\log2}{2}+\frac{\log3}{36}
       +\frac54\sum_{\substack{p\ {\rm prime}\\p\ {\rm odd}}}
                    \frac{\log p}{p^2}\\
 &<\frac7{20}+\frac{11}{360}+\frac7{16}
   =\frac{589}{720}<\frac{33}{40}.
 \end{aligned}
\]
Here the first inequality isolates $p=2,3$ and uses
$p/(p-1)\le5/4$ for $p\ge5$.
For real $x\ge y\ge100$, \eqref{prelim:von-mangoldt-bounds} gives
\[
 A(x)-A(y)\ge\log(x/y)-\log4-\frac{1+\log100}{100}.
\]
Indeed, the term $1/n$ absorbs $\log x-\log n$ for
$n=\lfloor x\rfloor$, and $(1+\log t)/t$ decreases for $t\ge100$.
Put $W(t)=A(t)-M(t)$. Then $W(x)\le V$, whereas the terms $4,9\le y$
give
\[
 W(y)\ge\frac{\log2}{4}+\frac{\log3}{9}.
\]
Consequently $M(x)-M(y)>\log(x/y)-2$, since the loss is less than
\[
 \frac{33}{40}+\frac74\log2-\frac19\log3
       +\frac{1+\log100}{100}
 <\frac{33}{40}+\frac{49}{40}-\frac{13}{108}+\frac3{50}
 =2-\frac7{675}<2.
\]
This proves \eqref{prelim:mertens-interval-bound}.
\end{proof}

\section{Integer counting and polynomial congruences}
\label{sec:arithmetic}

We prove a uniform prime-sum bound by coefficient counting and
polynomial congruences. Every polynomial used in the construction
has integer coefficients, and all polynomial congruences are
coefficientwise. No hypothesis about zeros is imposed.

\begin{proposition}\label{arith:main}
Let $\chi$ be a primitive quadratic Dirichlet character of
conductor $q\geq3$.  For every real $x\geq2$,
\begin{equation}\label{arith:target}
 S_-(x,\chi)\leq
 \frac{11}{20}\log x+\frac{29}{60}\log q+6.
\end{equation}
\end{proposition}

\subsection{A counting lemma}

The following lemma abstracts the local image counts and the
coefficient-box argument in \cite[Lemmas~3--5]{Kerr2026}.

\begin{lemma}\label{arith:counting}
Let $n,m$ be positive integers and let
$\iota :\mathbb Z^n\longrightarrow\mathbb Z^m$ be an injective additive
map, and suppose that $\lVert \iota (e_i)\rVert_\infty\leq w_i$, where
$e_1,\ldots,e_n$ are the standard basis vectors and $w_i>0$.
Let $\mathcal P$ be a nonempty finite set of primes.  Suppose that for
each $p\in\mathcal P$ there is an integer basis
$b_1^{(p)},\ldots,b_n^{(p)}$ of $\mathbb Z^n$ and nonnegative integers
$h_i(p)$ such that
\[
 \iota (b_i^{(p)})\in p^{h_i(p)}\mathbb Z^m.
\]
Then
\begin{equation}\label{arith:counting-bound}
 H =\sum_{p\in\mathcal P}\sum_{i=1}^n h_i(p)\log p
 \leq n\log n+\sum_{i=1}^n\log w_i.
\end{equation}
\end{lemma}

\begin{proof}
Choose an integer $\nu >\max_{p,i}h_i(p)$ and put
${d_0}=\prod_{p\in\mathcal P}p^\nu $.  In the basis belonging to $p$, changing
the $i$th coefficient by a multiple of $p^{\nu -h_i(p)}$ leaves its image
unchanged modulo $p^\nu $. Thus
\[
 \#\operatorname{im}(\iota \bmod p^\nu )
 \le\prod_{i=1}^n p^{\nu -h_i(p)}
 =p^{\nu n-\sum_i h_i(p)}.
\]
The Chinese remainder theorem gives
\[
 \#\operatorname{im}(\iota \bmod {d_0})
 \le\prod_{p\in\mathcal P}p^{\nu n-\sum_i h_i(p)}
 ={d_0}^n e^{-H }.
\]
Consider the half-open integer box
\[
 \mathcal B=\left\{c\in\mathbb Z^n:
       0\leq c_i<\frac{d_0}{nw_i}\ \ (1\leq i\leq n)\right\}.
\]
Since $[0,t)$ contains $\lceil t\rceil\ge t$ integers,
\[
 \#\mathcal B
 =\prod_{i=1}^n\left\lceil\frac{d_0}{nw_i}\right\rceil
 \ge\frac{{d_0}^n}{n^n\prod_iw_i}.
\]
For $c,c'\in\mathcal B$ with $\iota (c)\equiv \iota (c')\pmod {d_0}$,
\[
 \iota (c-c')\in {d_0}\mathbb Z^m,\qquad
 \|\iota (c-c')\|_\infty
 \le\sum_iw_i|c_i-c_i'|<{d_0}.
\]
Hence $\iota (c-c')=0$, and injectivity gives $c=c'$. Consequently
\[
 \frac{{d_0}^n}{n^n\prod_iw_i}
 \le\#\mathcal B
 \le\#\operatorname{im}(\iota \bmod {d_0})
 \le {d_0}^n e^{-H },
\]
which is equivalent to \eqref{arith:counting-bound}.
\end{proof}

\subsection{Remainders and an integer evaluation map}

Fix the squarefree integer $d\ne1$ in
\eqref{prelim:quadratic-reduction}, with $|d|\le q$.  For $f(t)\in\mathbb Z[t]$, let $\rho(f)$ be its remainder
on division by the monic polynomial $t^2-d$.  Thus
$\rho(f)=a+bt$ with $a,b\in\mathbb Z$.  In particular,
\begin{equation}\label{arith:remainder-product}
 \rho\bigl((a+bt)(c+et)\bigr)
       =(ac+dbe)+(ae+bc)t.
\end{equation}
For $f,g\in\mathbb Z[t]$, remainder reduction gives
\[
 \rho(f+g)=\rho(f)+\rho(g),\qquad
 \rho(fg)=\rho\bigl(\rho(f)\rho(g)\bigr).
\]
For a linear remainder put
\[
 \lVert a+bt\rVert=|a|+\sqrt{|d|}\,|b|.
\]
The multiplication formula shows that this size is submultiplicative:
\[
 \lVert\rho(fg)\rVert
 \leq\lVert\rho(f)\rVert\lVert\rho(g)\rVert.
\]
Since $|d|\geq1$, it bounds the absolute value of each of the two
integer coefficients.

For now, let ${p_0} $ be an odd prime with $\chi({p_0} )=-1$,
${p_0} \ge10000$, and ${p_0} \ge4q$. Set
\begin{equation}\label{arith:parameters}
 {s_0}=\left\lfloor\frac {p_0} {30}\right\rfloor,
 \qquad n={p_0} {s_0},
 \qquad {m}=\lceil\sqrt {p_0} \rceil,
 \qquad {b_0}=2\sqrt{{p_0} q}.
\end{equation}
We use every integer pair in the grid
\[
 0\leq u<{m},\qquad 1\leq v<{m}.
\]
It contains ${m}({m}-1)$ points.  Notice that
$2{m}({m}-1)\geq {m}^2>{p_0} $, since ${m}\geq2$ and a prime is not a square.
Also ${m}-1<\sqrt {p_0} $ and $2({m}-1)<{p_0} $.

Let $\mathcal V$ consist of the integer-coefficient polynomials
\[
 f(z,w)=\sum_{\substack{0\leq r<{p_0} \\0\leq s<{s_0}}}c_{r,s}z^rw^s,
 \qquad c_{r,s}\in\mathbb Z,
\]
identified with $\mathbb Z^n$ by their coefficient arrays.  Define
\begin{equation}\label{arith:evaluation}
 \iota :\mathcal V\longrightarrow\mathbb Z^{2{m}({m}-1){s_0}}
\end{equation}
by listing the two coefficients of every remainder
\[
 \rho\bigl((\partial_z^jf)(u+vt,u-vt)\bigr),
 \qquad 0\leq u<{m},\quad 1\leq v<{m},\quad 0\leq j<{s_0}.
\]
All derivatives are ordinary formal derivatives, without division by
factorials.

The next proof is a polynomial-remainder formulation of
\cite[proof of Theorem~1, Nonvanishing]{vanDijk2026};
compare also \cite[Lemma~4]{Stepanov2026}.

\begin{lemma}\label{arith:injectivity}
The integer map $\iota $ is injective.
\end{lemma}

\begin{proof}
Suppose that $\iota (f)=0$ for some nonzero $f\in\mathcal V$.  Divide its
coefficients by their largest common power of ${p_0} $.  The resulting
polynomial, again called $f$, still satisfies $\iota (f)=0$, and at least
one coefficient is not divisible by ${p_0} $.

For each grid point put
\[
 f_{u,v}(\xi )=(\xi -u)^2-dv^2.
\]
Modulo ${p_0} $ these are irreducible quadratics, because $d$ is a nonsquare
and $v$ is nonzero modulo ${p_0} $.  They are pairwise different: equality
of two such polynomials first gives $u=u'$ modulo ${p_0} $ and then
$v^2=(v')^2$ modulo ${p_0} $.  The specified ranges of $u,v$, together with
$v+v'\leq2({m}-1)<{p_0} $, imply $u=u'$ and $v=v'$ as integers.  Thus the
polynomials $f_{u,v}$ are pairwise relatively prime modulo ${p_0} $.

For $0\leq j<{s_0}$, define the integer polynomial
\[
 h_{u,v,j}(\xi )=(\partial_z^jf)(\xi ,2u-\xi ).
\]
We first show that $f_{u,v}$ divides $h_{u,v,j}$ already in
$\mathbb Z[\xi ]$.  Divide by the monic quadratic and write the remainder
as $a\xi +b$.  After substituting $\xi =u+vt$ and reducing modulo $t^2-d$,
the dividend has remainder zero by $\iota (f)=0$, while the remainder of
the divisor is zero.  Hence
\[
 (au+b)+avt=0.
\]
Since $v>0$, this forces $a=b=0$ and proves the assertion.

Now reduce polynomial coefficients modulo ${p_0} $, and put
\[
 g(\xi )=f(\xi ,\xi ^{p_0} )\pmod {p_0} .
\]
The exponents $r+{p_0} s$, with $0\leq r<{p_0} $ and $0\leq s<{s_0}$, are distinct.
Consequently $g\ne0$ and $\deg g\leq {p_0} {s_0}-1$.
Euler's criterion and Fermat's little theorem give the polynomial
congruence
\begin{equation}\label{arith:quadratic-congruence}
 \xi ^{p_0} \equiv u+(dv^2)^{({p_0} -1)/2}(\xi -u)
      \equiv2u-\xi \pmod{{p_0} ,f_{u,v}(\xi )}.
\end{equation}
Here the first congruence uses $(\xi -u)^2\equiv dv^2$ modulo $f_{u,v}$.
Every positive derivative of $\xi ^{p_0} $ is zero modulo ${p_0} $, so
\[
 g^{(j)}(\xi )\equiv(\partial_z^jf)(\xi ,\xi ^{p_0} )\pmod {p_0} .
\]
Together with \eqref{arith:quadratic-congruence} and the divisibility
of $h_{u,v,j}$, this proves
\[
 f_{u,v}\mid g^{(j)}\pmod {p_0} \qquad(0\leq j<{s_0}).
\]

These divisibilities imply $f_{u,v}^{s_0}\mid g\pmod{p_0}$ by
polynomial division modulo ${p_0} $: if
$g\equiv f_{u,v}^k{k_0}\pmod{p_0}$ and $0\le k<{s_0}$,
the product rule gives
\[
 g^{(k)}\equiv k!(f_{u,v}')^k{k_0}\pmod{p_0,f_{u,v}}.
\]
The integer $k!$ is invertible modulo ${p_0} $, since $k<{s_0}<{p_0} $.
Furthermore $f_{u,v}$ and $f_{u,v}'=2(\xi -u)$ are relatively prime
modulo ${p_0} $, because $dv^2\not\equiv0\pmod {p_0} $.  The polynomial
Euclidean algorithm therefore gives $f_{u,v}\mid {k_0}\pmod{p_0}$.
Starting with $k=0$ and proceeding inductively proves the assertion.

Pairwise relative primality now gives
\[
 \prod_{\substack{0\leq u<{m}\\1\leq v<{m}}}f_{u,v}(\xi )^{s_0}\mid g(\xi )
 \pmod {p_0} .
\]
The polynomial on the left is monic of degree
$2{m}({m}-1){s_0}$. Since $g\ne0$, we obtain the contradiction
\[
 2{m}({m}-1){s_0}\le\deg g\le {p_0} {s_0}-1<2{m}({m}-1){s_0}.
\]
\end{proof}

\subsection{Triangular bases and divisibility}

Put
\[
 \mathcal P=\{p\leq {p_0} :p\text{ is odd and }\chi(p)=-1\},
 \qquad
 h_{r,s}(p)=\min\left(\left\lfloor\frac rp\right\rfloor,{s_0}-1-s\right),
\]
and write
\begin{equation}\label{arith:H-definition}
 H _p=\sum_{\substack{0\leq r<{p_0} \\0\leq s<{s_0}}}h_{r,s}(p),
 \qquad H =\sum_{p\in\mathcal P}H _p\log p.
\end{equation}
For each $p\in\mathcal P$, replace the monomials by the polynomials
\begin{equation}\label{arith:adapted-basis}
 g_{r,s}^{(p)}(z,w)=z^{r-ph_{r,s}(p)}w^s(z^p-w)^{h_{r,s}(p)}.
\end{equation}
Writing $h=h_{r,s}(p)$, their expansion is
\[
 g_{r,s}^{(p)}=z^rw^s+
 \sum_{k=1}^h(-1)^k\binom hk
             z^{r-pk}w^{s+k}.
\]
Every monomial stays in the original rectangle, and every additional
term has first exponent smaller than $r$.  Successive substitution in
increasing order of $r$ shows that these polynomials form an integer
basis of $\mathcal V$.  This is the triangular change of basis in
Green~\cite[Lemma~5.1]{Green2026}.

Euler's criterion and the binomial theorem give, coefficientwise,
\begin{equation}\label{arith:frobenius}
 \rho\bigl((u+vt)^p\bigr)\equiv u-vt\pmod p,
\end{equation}
because $\rho(t^p)=d^{(p-1)/2}t\equiv-t\pmod p$.
Thus the remainder of $z^p-w$ at every grid point has both coefficients
divisible by $p$. For $\psi_p(z,w)=z^p-w$,
\[
 \partial_z^j\psi_p(z,w)=
 \begin{cases}
 z^p-w,&j=0,\\
 \dfrac{p!}{(p-j)!}z^{p-j},&1\le j\le p,\\
 0,&j>p.
 \end{cases}
\]
Thus every derivative has a remainder divisible by $p$ at each grid
point.
In every term of the product rule applied to
\eqref{arith:adapted-basis}, each of its $h$ factors $z^p-w$ therefore
contributes a factor of $p$ after evaluation and remainder reduction.
Consequently
\begin{equation}\label{arith:divisibility}
 \iota (g_{r,s}^{(p)})\in p^{h_{r,s}(p)}\mathbb Z^{2{m}({m}-1){s_0}}.
\end{equation}
This argument includes all derivative orders used in $\iota $, even when
they are at least $p$.

On the grid,
\[
 \lVert u\pm vt\rVert
 \leq({m}-1)(1+\sqrt{|d|})
 <\sqrt {p_0} (1+\sqrt q)\leq {b_0}.
\]
For $j\leq r$ the corresponding derivative of $z^rw^s$ therefore
has remainder size at most ${p_0} ^j{b_0}^{r+s-j}$; it is zero for $j>r$.
Since ${p_0} \ge4q$ gives ${b_0}\le {p_0} $, we obtain the coordinate bound
\[
 \lVert \iota (z^rw^s)\rVert_\infty
 \leq({p_0} /{b_0})^{{s_0}-1}{b_0}^{r+s}=w_{r,s}.
\]
Apply Lemma~\ref{arith:counting}, using
Lemma~\ref{arith:injectivity} and \eqref{arith:divisibility}.
The identity $\sum_{r,s}(r+s)={p_0} {s_0}({p_0} +{s_0}-2)/2$ gives
\[
 H \leq {p_0} {s_0}\log({p_0} {s_0})+{p_0} {s_0}({s_0}-1)\log({p_0} /{b_0})
                   +\frac{{p_0} {s_0}({p_0} +{s_0}-2)}2\log {b_0}.
\]
Put $\lambda=({s_0}-1)/({p_0} -1)\le1/30$. Since ${p_0} {s_0}\le {p_0} ^2$,
\[
 \begin{aligned}
 \frac{2H }{{p_0} {s_0}({p_0} -1)}
 &\le\frac{2\log({p_0} {s_0})}{{p_0} -1}
    +\frac{2({s_0}-1)}{{p_0} -1}\log {p_0} 
    +\frac{{p_0} -{s_0}}{{p_0} -1}\log {b_0}\\
 &\le\frac{4\log {p_0} }{{p_0} -1}+\log {b_0}
       +\lambda(2\log {p_0} -\log {b_0}).
 \end{aligned}
\]
The coefficient of $\lambda$ is positive. Substituting
$\lambda\le1/30$ and
$\log {b_0}=\tfrac12\log {p_0} +\tfrac12\log q+\log2$ gives
\begin{equation}\label{arith:H-upper}
 \frac{2H }{{p_0} {s_0}({p_0} -1)}
 \leq\frac{11}{20}\log {p_0} +\frac{29}{60}\log q
       +\frac{29}{30}\log2+\frac{4\log {p_0} }{{p_0} -1}.
\end{equation}

\subsection{Summing the exponents}

Put ${U}=\sum_{p\in\mathcal P}(\log p)/p$.
For every integer $k\geq0$,
\begin{equation}\label{arith:min-sum}
 2\sum_{s=0}^{{s_0}-1}\min(k,{s_0}-1-s)\geq(2{s_0}-1)k-k^2.
\end{equation}
For $k<{s_0}$ this is an equality.  For $k\geq {s_0}$, the difference
between the two sides is $(k-{s_0})(k-{s_0}+1)\geq0$.
If ${p_0} =ap+b$ with $0\leq b<p$, summation in constant blocks gives
\begin{align}
 \sum_{r=0}^{{p_0} -1}\left\lfloor\frac rp\right\rfloor
 &=\frac{pa(a-1)}2+ab
  =\frac{{p_0} ({p_0} -p)}{2p}+\frac{b(p-b)}{2p}\notag\\
 &\geq\frac{{p_0} ({p_0} -1)}{2p}-\frac {p_0} 2,
 \label{arith:floor-sum}\\
 \sum_{r=0}^{{p_0} -1}\left\lfloor\frac rp\right\rfloor^2
 &\leq\frac{{p_0} ({p_0} -1)(2{p_0} -1)}{6p^2}.
 \label{arith:floor-square}
\end{align}
Apply \eqref{arith:min-sum} with $k=\lfloor r/p\rfloor$ and use
\eqref{arith:floor-sum}--\eqref{arith:floor-square}.  After discarding
a nonnegative term ${p_0} /2$, this gives
\[
 2H _p\geq
 \frac{{s_0}{p_0} ({p_0} -1)}p-\frac{{p_0} ({p_0} -1)}{2p}-{s_0}{p_0} 
 -\frac{{p_0} ({p_0} -1)(2{p_0} -1)}{6p^2}.
\]
Multiply by $\log p$ and sum over $\mathcal P$.  For the negative
prime sums, use $\vartheta({p_0} )$ and
\eqref{prelim:odd-square-bound}, respectively.  We obtain
\begin{equation}\label{arith:H-lower}
 \frac{2H }{{p_0} {s_0}({p_0} -1)}
 \geq {U}-\frac{U}{2{s_0}}-\frac{\vartheta({p_0} )}{{p_0} -1}
          -\frac{7(2{p_0} -1)}{120{s_0}}.
\end{equation}

\subsection{The remainder and the small cases}

\begin{proof}[Proof of Proposition~\ref{arith:main}]
If an odd prime $p\leq x$ contributes to $S_-(x,\chi)$, let ${p_0} $ be
the largest such prime. First suppose that ${p_0} \ge10000$ and ${p_0} \ge4q$.
The odd-prime
part of $S_-(x,\chi)$ equals ${U}$, and the possible contribution of $2$
is at most $(\log2)/2$.  Combining
\eqref{arith:H-upper} and \eqref{arith:H-lower} gives
\[
 S_-(x,\chi)\leq
 \frac{11}{20}\log {p_0} +\frac{29}{60}\log q+{\mathcal R}({p_0} ),
\]
where
\begin{equation}\label{arith:remainder}
 {\mathcal R}({p_0} )=
 \frac{\vartheta({p_0} )+4\log {p_0} }{{p_0} -1}
 +\frac{U}{2{s_0}}+\frac{7(2{p_0} -1)}{120{s_0}}+\frac{22}{15}\log2.
\end{equation}
Here ${s_0}\ge333$ and ${p_0} <30({s_0}+1)$, so
${p_0} /{s_0}<30+30/333<301/10$.
By Lemma~\ref{prelim:prime-estimates} and
\eqref{prelim:numerical},
\[
 \vartheta({p_0} )\leq(\log4){p_0} ,\qquad
 {U}\leq\log {p_0} +\log4,\qquad
 \log2<\frac7{10},\qquad\log10000<10.
\]
The functions $t/(t-1)$, $(\log t)/(t-1)$ and
$(\log t+\log4)/t$ are decreasing for $t\ge10000$.
Thus the four terms in \eqref{arith:remainder} satisfy
\begin{align*}
 \frac{\vartheta({p_0} )+4\log {p_0} }{{p_0} -1}
 &<\frac{10000(7/5)+4\cdot10}{9999}
   =\frac{14040}{9999}<\frac{141}{100},\\
 \frac{U}{2{s_0}}
 &<\frac{(301/10)(10+7/5)}{20000}<\frac1{50},\\
 \frac{7(2{p_0} -1)}{120{s_0}}&<\frac{2107}{600}<\frac{53}{15},
 \qquad
 \frac{22}{15}\log2<\frac{77}{75}.
\end{align*}
Consequently
\begin{equation}\label{arith:remainder-six}
 {\mathcal R}({p_0} )<\frac{141}{100}+\frac1{50}
                 +\frac{53}{15}+\frac{77}{75}
                 =\frac{599}{100}<6.
\end{equation}
Since ${p_0} \leq x$, this proves \eqref{arith:target} in this case.

If the largest contributing odd prime satisfies ${p_0} <10000$, then the
entire prime sum, including the possible contribution of $2$, is at
most $\log {p_0} +\log4$ by \eqref{prelim:prime-upper-bound}.  Hence
\[
 S_-(x,\chi)-\frac{11}{20}\log x-\frac{29}{60}\log q
 \leq\frac9{20}\log {p_0} +\log4
 <\frac9{20}\cdot10+\frac75=\frac{59}{10}<6.
\]
If instead ${p_0} <4q$, the same elementary upper bound gives
\[
 S_-(x,\chi)-\frac{11}{20}\log x-\frac{29}{60}\log q
 <-\frac1{30}\log {p_0} +\frac{89}{60}\log4
 <\frac{623}{300}<6.
\]
Finally, if there is no contributing odd prime, then
$S_-(x,\chi)\leq(\log2)/2$, and the same conclusion is immediate.
\end{proof}

The construction retains the quadratic relation $t^2=d$, expressed
through a fixed polynomial remainder. The factor-count proof of
Lemma~\ref{arith:injectivity} preserves the auxiliary-prime mechanism
and expresses it entirely through polynomial division and degree
counting.

\section{A real-variable estimate at a real zero}
\label{real:section}

The estimate below is an explicit refinement, at a real zero, of the
standard weighted Dirichlet hyperbola argument; compare
\cite[Exercise~58]{Tao2014} and \cite[Section~1, Lemma~5]{Tao2015}.
We combine the two character tails before estimating them, preserving
the cancellation between their kernels. This gives an explicit
remainder bound uniform as $\beta$ approaches $1$, without an extra
factor $(1-\beta)^{-1}$ in the error. Throughout this section, $\chi$
is a nonprincipal quadratic character modulo $q\ge3$; primitivity is not
needed.

For $0\le {u}\le v$, write
\[
 A_{u}(v)=\sum_{{u}<n\le v}\chi(n).
\]
The complete character sum vanishes
\cite[Section~6.9]{Apostol}. Write
$\lfloor v\rfloor-\lfloor {u}\rfloor=kq+r$, where $0\le r<q$.
Removing $k$ complete periods and comparing the remaining block
with its complement gives
\[
 |A_{u}(v)|\le\min(r,q-r)\le\frac q2.
\]
Convergence for $s>0$ and partial summation
\cite[Theorems~6.17 and~4.2]{Apostol} therefore give, for ${u}\ge1$,
\begin{equation}\label{real:tail}
 R_s({u})=\sum_{n>{u}}\frac{\chi(n)}{n^s}
 =s\int_{u}^\infty A_{u}(v)v^{-s-1}\,dv,
 \qquad |R_s({u})|\le\frac q2 {u}^{-s}.
\end{equation}

Put
\begin{equation}\label{real:positive-sum}
 a(n)=(1*\chi)(n)=\sum_{d\mid n}\chi(d),\qquad
 H_\beta(t)=\sum_{n\le t}\frac{a(n)}{n^\beta}
 \quad(t\ge1).
\end{equation}
By the standard properties of Dirichlet convolution and real
characters \cite[Theorems~2.14 and~6.19]{Apostol}, $a$ is
multiplicative, and
\[
 a(p^k)=\sum_{j=0}^k\chi(p)^j=
 \begin{cases}
 k+1,&\chi(p)=1,\\
 1,&\chi(p)=0,\\
 \dfrac{1+(-1)^k}{2},&\chi(p)=-1.
 \end{cases}
\]
Consequently $a(n)\ge0$ for every $n$, and
\begin{equation}\label{real:positive-normalization}
 a(1)=1,\qquad a(p)=1+\chi(p),\qquad
 a(m^2)\ge1,\qquad H_\beta(t)\ge1.
\end{equation}

\begin{proposition}[Paired tails]\label{real:hyperbola}
Suppose that $1/2\leq\beta<1$ and $L(\beta,\chi)=0$. Write
$\varepsilon=1-\beta$. For every real $t\geq1$,
\begin{equation}\label{real:hyperbola-bound}
 H_\beta(t)=\frac{L(1,\chi)}{\varepsilon}t^\varepsilon+E(t),
 \qquad
 |E(t)|\leq
 \left(q+1+\frac12\log t\right)t^{1/2-\beta}.
\end{equation}
\end{proposition}

\begin{proof}
Set ${u}=\sqrt t$. The Dirichlet hyperbola identity
\cite[Theorem~3.17]{Apostol}, applied to $dm\le t$, gives
\begin{equation}\label{real:hyperbola-identity}
 \begin{split}
 H_\beta(t)
 ={}&\sum_{d\leq {u}}d^{-\beta}
       \sum_{m\leq t/d}\frac{\chi(m)}{m^\beta}\\
 &+\sum_{m\leq {u}}\frac{\chi(m)}{m^\beta}
       \sum_{{u}<d\leq t/m}d^{-\beta}.
 \end{split}
\end{equation}
The first line has absolute value at most
\begin{equation}\label{real:first-half}
 \sum_{d\leq {u}}d^{-\beta}\frac q2(t/d)^{-\beta}
 \leq\frac q2t^{1/2-\beta},
\end{equation}
by $L(\beta,\chi)=0$ and \eqref{real:tail}.

Euler summation \cite[Theorem~3.1]{Apostol}, valid for real endpoints,
gives
\[
 \sum_{{u}<n\le {v}}n^{-\beta}
 =\frac{{v}^\varepsilon-{u}^\varepsilon}{\varepsilon}
    +\mathcal R_\beta({u},{v}),\qquad
 |\mathcal R_\beta({u},{v})|\le {u}^{-\beta}\qquad({v}\ge {u}).
\]
Indeed, with $\{w\}=w-\lfloor w\rfloor$, the remainder is
\[
 \mathcal R_\beta({u},{v})
 =\{{u}\}{u}^{-\beta}-\{{v}\}{v}^{-\beta}
      -\beta\int_{u}^{v}\{w\}w^{-\beta-1}\,dw,
\]
which lies between $-{u}^{-\beta}$ and ${u}^{-\beta}$.
Applying this to the second line of \eqref{real:hyperbola-identity}
yields
\[
 \frac{t^\varepsilon}{\varepsilon}
      \sum_{m\leq {u}}\frac{\chi(m)}m
 -\frac{{u}^\varepsilon}{\varepsilon}
      \sum_{m\leq {u}}\frac{\chi(m)}{m^\beta}
 +E_0,
\]
where
\begin{equation}\label{real:rounding-error}
 |E_0|
 \leq {u}^{-\beta}\sum_{m\leq {u}}m^{-\beta}
 \leq {u}^{\varepsilon-\beta}(1+\log {u})
 =\left(1+\frac12\log t\right)t^{1/2-\beta}.
\end{equation}
Here $m^{-\beta}\leq {u}^\varepsilon/m$ for $m\leq {u}$.
Since $L(\beta,\chi)=0$, the two main sums equal
\begin{equation}\label{real:paired-tails}
 \frac{L(1,\chi)}{\varepsilon}t^\varepsilon
 -\frac{t^\varepsilon}{\varepsilon}R_1({u})
 +\frac{{u}^\varepsilon}{\varepsilon}R_\beta({u}).
\end{equation}

The last two terms are estimated together. We combine the two
convergent tails with a common finite cutoff before passing to the
limit. Their sum is
\[
 E_1={u}^\varepsilon\sum_{m>{u}}\chi(m)g(m),
 \qquad
 g(v)=\frac{v^{-\beta}-{u}^\varepsilon v^{-1}}{\varepsilon}
     =\frac1v\int_{u}^v w^{-\beta}\,dw.
\]
For $v\ge u$, the kernel satisfies
\[
 g(v)\ge0,\qquad g({u})=\lim_{v\to\infty}g(v)=0,\qquad
 g'(v)=\frac{v^{-2}}{\varepsilon}
        \left({u}^\varepsilon-\beta v^\varepsilon\right).
\]
The derivative changes sign once, at $v={u}\beta^{-1/\varepsilon}$.
For $x=v/{u}\geq1$, the assumption $\beta\geq1/2$ gives
\[
 g(v)
 =\frac{{u}^{-\beta}}x\int_1^x w^{-\beta}\,dw
 \leq2{u}^{-\beta}\frac{\sqrt x-1}{x}
 \leq\frac12{u}^{-\beta},
\]
where the last inequality is $(\sqrt x-2)^2\geq0$. Therefore
\[
 \int_{u}^\infty|g'(v)|\,dv=2\max_{v\geq {u}}g(v)\leq {u}^{-\beta}.
\]
Another summation by parts, whose boundary terms vanish, yields
\[
 \sum_{m>{u}}\chi(m)g(m)
 =-\int_{u}^\infty A_{u}(v)g'(v)\,dv.
\]
The bound $|A_{u}(v)|\le q/2$ now gives
\begin{equation}\label{real:paired-error}
 |E_1|\leq\frac q2 {u}^{\varepsilon-\beta}
 =\frac q2t^{1/2-\beta}.
\end{equation}
Combining the three errors gives
\[
 |E(t)|
 \le\left(\frac q2+1+\frac12\log t+\frac q2\right)t^{1/2-\beta},
\]
as required.
\end{proof}

Pairing the tails in \eqref{real:paired-tails} removes the factor
$\varepsilon^{-1}$ that their separate absolute estimates would
introduce. The coefficient $L(1,\chi)/\varepsilon$ will be eliminated
by comparing finite positive sums; no lower bound for it is needed.
For comparison, the standard bound in
\cite[Exercise~58]{Tao2014}, specialized to a real zero, has remainder
$O(q(1-\beta)^{-1}t^{1/2-\beta})$. The estimate above retains the
cancellation needed to avoid that additional dependence on the zero.

\section{Positivity and the explicit zero gap}
\label{proof:section}

The use of $a=1*\chi$, the harmonic sum
$H_1(y)=\sum_{n\le y}a(n)/n$, and a single-large-prime comparison
already appears in \cite[proof of Proposition~3.5]{TaoTeravainen2022}.
That proof uses a lower bound for $L(1,\chi)$ and an estimate for
$L'/L(1,\chi)$; see also \cite[Section~1, Lemma~5]{Tao2015} for
the positive-convolution approach to prime bias.
The lemma below uses the logarithmic weight $\log p/p$ and retains
the full error terms. Its Abel summation identity, combined with
Proposition~\ref{real:hyperbola}, supplies the explicit comparison
without those preliminary estimates. This realizes the standard
linear dependence on $1-\beta$ noted in the introduction with the
constants needed for a single choice of parameters for all conductors.

\begin{lemma}[A logarithmically weighted comparison]
\label{real:large-prime}
Under the hypotheses of Proposition~\ref{real:hyperbola}, write
$\varepsilon=1-\beta$ and put
\[
 H_1(t)=\sum_{n\le t}\frac{a(n)}n,\qquad
 B(t)=\left(q+1+\frac12\log t\right)t^{-1/2}.
\]
Let $1\le y\le x$ and suppose that $\varepsilon\log(xy)\le1$.
Set $\kappa =\log x\log y+\tfrac12(\log x)^2$. Then
\begin{align}
 J=\sum_{y<p\le x}\frac{a(p)\log p}{p}
 &\le\varepsilon \kappa +\frac{\mathcal E (x,y)}{H_1(y)},
 \label{real:prime-mass}\\
 \mathcal E (x,y)
 &=B(y)(\varepsilon \kappa +\log y)+B(xy)\log(xy)\notag\\
 &\quad+2B(y)+2y^{-1/2}.
 \label{real:comparison-error}
\end{align}
If also $y\ge\max(q,100)$, then
\begin{equation}\label{real:inert-lower}
 S_-(x,\chi)\ge\log(x/y)-2-\frac J2.
\end{equation}
\end{lemma}

\begin{proof}
Put $N(t)=\sum_{n\le t}a(n)(\log n)/n$.
For $n\le y<p\le x$, we have $(n,p)=1$ and $y<np\le xy$.
The map $(n,p)\mapsto np$ is injective: if $np=n'p'$ with $p\ne p'$,
then $p\mid n'\le y<p$, a contradiction. Multiplicativity,
nonnegativity, and $\log p\le\log(np)$ give
\begin{equation}\label{real:weighted-injection}
 \begin{aligned}
 H_1(y)J
 &=\sum_{\substack{n\le y\\y<p\le x}}
       \frac{a(np)\log p}{np}\\
 &\le\sum_{y<m\le xy}\frac{a(m)\log m}{m}
  =N(xy)-N(y).
 \end{aligned}
\end{equation}

Write $D(t)=t^{-\varepsilon}E(t)$ and
$\phi(t)=(\log t)t^{-\varepsilon}$. Partial summation gives
\[
 \begin{aligned}
 N(xy)-N(y)
 ={}&\phi(xy)H_\beta(xy)-\phi(y)H_\beta(y)\\
 &-\int_y^{xy}H_\beta(t)\phi'(t)\,dt.
 \end{aligned}
\]
Substituting \eqref{real:hyperbola-bound} yields exactly
\begin{equation}\label{real:weighted-abel}
 N(xy)-N(y)=L(1,\chi)\kappa +R,
\end{equation}
where
\[
 \begin{aligned}
 R={}&D(xy)\log(xy)-D(y)\log y\\
 &+\int_y^{xy}(\varepsilon\log t-1)\frac{D(t)}t\,dt.
 \end{aligned}
\]
Since $|D(t)|\le B(t)$ and $0\le\varepsilon\log t\le1$ on
$[y,xy]$,
\[
 |R|\le B(xy)\log(xy)+B(y)\log y
           +\int_y^\infty\frac{B(t)}t\,dt.
\]
The integral is $2B(y)+2y^{-1/2}$.

Positivity also gives $H_\beta(y)\le y^\varepsilon H_1(y)$, so
\[
 \frac{L(1,\chi)}{\varepsilon}
 =y^{-\varepsilon}\bigl(H_\beta(y)-E(y)\bigr)
   \le H_1(y)+B(y).
\]
Since $\varepsilon>0$ and $\kappa\ge0$, it follows that
\[
 L(1,\chi)\kappa
 \le\varepsilon \kappa \bigl(H_1(y)+B(y)\bigr).
\]
Combining this with \eqref{real:weighted-injection} and
\eqref{real:weighted-abel}, and dividing by $H_1(y)\ge1$, proves
\eqref{real:prime-mass}. No sign assumption or preliminary estimate
for $L(1,\chi)$ has been used.

If $y\ge q$, every prime $p>y$ is coprime to $q$, and $a(p)=2$
exactly when $\chi(p)=1$. Hence
\[
 S_-(x,\chi)-S_-(y,\chi)
 =M(x)-M(y)-\frac J2.
\]
Now use $S_-(y,\chi)\ge0$ and
\eqref{prelim:mertens-interval-bound}.
\end{proof}

\begin{proof}[Proof of Theorem~\ref{thm:main}]
Suppose that $L(\beta,\chi)=0$ for some $\beta\in(0,1)$ with
\[
 \delta=(1-\beta)\log q\le\frac1{100}.
\]
Take
\begin{equation}\label{proof:parameters}
 y=q^6,\qquad x=q^{70}.
\end{equation}
Since $\log q\ge\log3>1$, we have
$\beta\ge1-1/(100\log q)>1/2$. Moreover,
$\varepsilon\log(xy)=76\delta\le76/100<1$,
$y\ge729$, and $y>\max(q,100)$.
Thus Lemma~\ref{real:large-prime} applies.

The squares supply a uniform normalization: by
\eqref{real:positive-normalization},
\begin{equation}\label{proof:square-normalization}
 \begin{aligned}
 H_1(y)&\ge\sum_{m=1}^{27}\frac1{m^2}
 \ge\sum_{m=1}^7\frac1{m^2}+\int_8^{28}\frac{dt}{t^2}\\
 &=\frac{282431}{176400}>\frac85.
 \end{aligned}
\end{equation}
For $k\ge2$, put
\[
 f_k(q)=\left(q+1+\frac k2\log q\right)q^{-k/2}.
\]
Its derivative satisfies
\[
 f_k'(q)=q^{-k/2-1}
 \left[\left(1-\frac k2\right)q-\frac{k^2}{4}\log q\right]<0
 \qquad(q\ge3).
\]
Using \eqref{prelim:numerical} at $q=3$ therefore gives
\begin{align}
 B(y)&\le(4+3\log3)3^{-3}<\frac{73}{270},
 \label{proof:lower-end-error}\\
 B(xy)&\le(4+38\log3)3^{-38}
       <50\cdot3^{-38}<\frac1{10000},
 \label{proof:upper-end-error}\\
 \frac{2B(y)+2y^{-1/2}}{\log q}
 &<\left(\frac{73}{135}+\frac2{27}\right)\frac{12}{13}
   =\frac{332}{585}<\frac47.
 \label{proof:integral-error}
\end{align}
Here $\kappa =2870(\log q)^2$, whence
$\varepsilon \kappa /\log q=2870\delta\le287/10$.
Consequently, \eqref{real:comparison-error} satisfies the rational
certificate
\begin{equation}\label{proof:error-certificate}
 \begin{aligned}
 \frac{\mathcal E (x,y)}{\log q}
 &<\frac{73}{270}\left(\frac{287}{10}+6\right)
     +\frac{76}{10000}+\frac47\\
 &=\frac{1176629}{118125}
   =10-\frac{4621}{118125}<10.
 \end{aligned}
\end{equation}
Dividing by \eqref{proof:square-normalization} now gives
\begin{equation}\label{proof:prime-mass-certificate}
 \frac J{\log q}<\frac{287}{10}+\frac58\cdot10
 =\frac{699}{20}<35.
\end{equation}
By \eqref{real:inert-lower},
\begin{equation}\label{proof:lower-bound}
 S_-(q^{70},\chi)
 >\left(70-6-\frac{35}{2}\right)\log q-2
 =\frac{93}{2}\log q-2.
\end{equation}
On the other hand, Proposition~\ref{arith:main} gives
\begin{equation}\label{proof:upper-bound}
 S_-(q^{70},\chi)
 \le\left(\frac{11}{20}\cdot70+\frac{29}{60}\right)\log q+6
 =\frac{2339}{60}\log q+6.
\end{equation}
Combining these inequalities would imply
$\frac{451}{60}\log q<8$. This is impossible, since
\begin{equation}\label{proof:final-certificate}
 \frac{451}{60}\log q
 \ge\frac{451}{60}\log3
 >\frac{451}{60}\cdot\frac{13}{12}
 =8+\frac{103}{720}>8.
\end{equation}
The theorem follows.
\end{proof}

The same parameters apply to every conductor $q\ge3$.
The explicit arithmetic bound, the uniform paired-tail estimate, and
the logarithmically weighted comparison together supply the stated
constant without separate small-conductor zero computations. The
constant records this elementary calculation; no numerical optimality
is asserted.

\appendix
\section{Verification of the constant \texorpdfstring{$1/64$}{1/64}}
\label{app:constant}

Let $\chi$ be as in Theorem~\ref{thm:main}. Suppose that
$L(\beta,\chi)=0$ for some $\beta\in(0,1)$, and put
$\varepsilon=1-\beta$ and $\delta=\varepsilon\log q\le1/64$.
For $q\ge7$, take $x=q^{46}$ and $y=q^4$.
Then $\beta>1/2$, $\varepsilon\log(xy)=50\delta\le25/32<1$,
and $y\ge\max(q,100)$. Since $\sqrt y=q^2\ge49$, the square
normalization in \eqref{proof:square-normalization} still gives
$H_1(y)>8/5$. The monotonicity calculation used for
\eqref{proof:lower-end-error}, with exponents $4$ and $50$,
and \eqref{prelim:numerical} give
\begin{align*}
 B(y)&\le(8+2\log7)7^{-2}<\frac{17}{70},\\
 B(xy)&\le(8+25\log7)7^{-25}
        <57\cdot7^{-25}<\frac1{10000},\\
 y^{-1/2}&\le\frac1{49}.
\end{align*}
Here $\kappa=1242(\log q)^2$, so
$\varepsilon\kappa/\log q\le621/32$.
Consequently \eqref{real:comparison-error} gives
\begin{equation}\label{app:error-certificate}
 \begin{aligned}
 \frac{\mathcal E(x,y)}{\log q}
 &<\frac{17}{70}\left(\frac{621}{32}+4\right)
   +\frac{50}{10000}
   +\left(\frac{17}{35}+\frac2{49}\right)\frac{18}{35}\\
 &=6-\frac{21863}{548800}<6.
 \end{aligned}
\end{equation}
Thus \eqref{real:prime-mass} yields
$J/\log q<621/32+(5/8)6=741/32$.
Combining Proposition~\ref{arith:main} with
\eqref{real:inert-lower} would give
$(4453/960)\log q<8$, whereas
\[
 \frac{4453}{960}\log q
 >\frac{4453}{960}\cdot\frac{35}{18}
 =8+\frac{3523}{3456}>8.
\]

For $q=3,4$, the period blocks are respectively
$(3n+1)^{-s}-(3n+2)^{-s}>0$ and
$(4n+1)^{-s}-(4n+3)^{-s}>0$ for $n\ge0$ and $s>0$.
For $q=5$, strict convexity of $f(t)=t^{-s}$ gives
\[
 f(5n+1)-f(5n+2)>f(5n+3)-f(5n+4),
\]
so its period blocks are also positive.
Since the character series converge for $s>0$, this proves
$L(s,\chi)>0$ in these three cases. No primitive character has
conductor $6$. This proves the assertion in Remark~\ref{rem:constant}.

\paragraph{AI assistance.}
ChatGPT assisted with the development, organization, calculations,
and drafting of this note.

\end{document}